\documentclass[
final
]{dmtcs-episciences}

\usepackage[utf8]{inputenc}
\usepackage{subfigure}

\usepackage[square]{natbib}

\usepackage{amsmath}
\usepackage[capitalize,nameinlink]{cleveref}
\usepackage{thmtools}

\newcommand{\edgeh}[2]{\partial_e(#1,#2)}
\newcommand{\bdryh}[2]{\partial_v(#1,#2)}
\newcommand{\faceth}[1]{\mathcal{F}_{#1}}

\newcommand{\tw}{\operatorname{tw}}
\newtheorem{theorem}{Theorem}[section]
\newtheorem{definition}[theorem]{Definition}
\newtheorem{lemma}[theorem]{Lemma}
\newtheorem{corollary}[theorem]{Corollary}
\newtheorem{remark}[theorem]{Remark}

\author[Eppstein et al.]{David Eppstein\affiliationmark{1}\thanks{Research of David Eppstein was supported in part by NSF grant CCF-2212129.}
  \and Daniel Frishberg\affiliationmark{2}
  \and William Maxwell\affiliationmark{3}}
\title{On the expansion of Hanoi graphs}
\affiliation{
  University of California, Irvine, CA, United States\\
  California Polytechnic State University, San Luis Obispo, CA, United States\\
  Oregon State University, Corvallis, OR, United States}
\begin{document}
\publicationdata
{vol. 28:2}
{2026}
{31}
{10.46298/dmtcs.16770}
{2025-10-23; None}
{2026-05-02}

\maketitle
\begin{abstract}
The famous \emph{Tower of Hanoi} puzzle involves moving $n$ discs of distinct sizes from one of $p\geq 3$ pegs (traditionally $p=3$) to another of the pegs, subject to the constraints that only one disc may be moved at a time, and no disc can ever be placed on a disc smaller than itself. Much is known about the \emph{Hanoi graph} $H_p^n$, whose $p^n$ vertices represent the configurations of the puzzle, and whose edges represent the pairs of configurations separated by a single legal move. In a previous paper, the present authors presented nearly tight asymptotic bounds of $O((p-2)^n)$ and $\Omega(n^{(1-p)/2}(p-2)^n)$ on the \emph{treewidth} of this graph for fixed $p \geq 3$. In this paper we show that the upper bound is tight, by giving a matching lower bound of $\Omega((p-2)^n)$ for the \emph{expansion} of $H_p^n$.
\end{abstract}

\section{Introduction}
\label{sec:intro}

\subsection{Background and motivation}
In the well-studied \emph{Tower of Hanoi} puzzle, one is given $n$ discs $\{d_1, \dots, d_n\}$ of increasing size, sitting on the first of $p$ pegs. The object is to move all of the discs to some specified peg, while moving one disc at a time, and without ever putting one disc $d_i$ on a smaller disc $d_j$, $j < i$. The \emph{Hanoi graph} $H_p^n$ is the graph whose vertices correspond to the $p^n$ possible configurations of the puzzle, and whose edges correspond to the pairs of configurations that differ by a single valid move.

\emph{Treewidth} (see \cref{sec:exptw}) is a graph connectivity property with broad applications in algorithm design and analysis (see e.g.~\cite{FPTsurvey}). Intuitively, treewidth measures how difficult it is to recursively decompose a graph via small vertex separators. The treewidth of large graphs having a natural combinatorial definition\textemdash such as Hanoi graphs\textemdash is of independent theoretical interest. Some results are known e.g. for \emph{Kneser} graphs~\citep{harveywood} and their generalizations~\citep{genkneser}.

\emph{Expansion} is a graph connectivity property that in some sense is stronger than treewidth; it measures the robustness of a graph with respect to ``bottlenecks'', or sparse cuts. Expansion has a close connection to \emph{multicommodity flow} problems via a kind of flow-cut duality. Prior work on the expansion of large graphs~\citep{babaiszegedy, eppfrishtri, kaibelexp} connects closely to the study of the \emph{mixing} properties of Markov chain Monte Carlo algorithms for counting and sampling problems~\citep{sinclair_1992, levin2017markov} (see also, e.g., \citep{anari2}) via the relationship between (combinatorial) expansion and \emph{spectral} expansion. Algorithmic questions involving expansion are also of interest (see e.g.~\cite{arv}).

Certain graph properties of the special case $p = 3$ of $H_p^n$ have been somewhat easier to analyze than the case $p \geq 4$. The Hanoi graph $H_3^n$ on three pegs is known to have treewidth exactly four~\citep{hanoitw}, and the diameter of $H_3^n$ is $2^n - 1$~\citep{cartesian}. On the other hand, for $p \geq 4$, only (nearly tight) asymptotic bounds are known on the treewidth~\citep{hanoitw}, and the diameter of $H_p^n$ is unknown for $p \geq 4$~\citep{cartesian} although values have been computed for small~$n$ when $p = 4$, and compared with theoretical bounds~\citep{mythsmaths}.

\subsection{Markov chain decomposition}
Our technique builds on existing decomposition techniques developed in the analysis of Markov chain Monte Carlo (MCMC) algorithms for sampling problems. These algorithms are simple random walks that converge to a target distribution of interest. The \emph{mixing time} of such an algorithm\textemdash the time it takes to converge to the target distribution\textemdash is of interest, and often difficult to bound. A common means of obtaining such a bound is to use a standard theorem that relates the \emph{Cheeger constant}\textemdash essentially the expansion of the state space of the chain, viewed as a graph, up to normalization factors\textemdash to the mixing time.  

In a prior work~\citep{eppfrishtri}, motivated by the question of the mixing time of a well-studied~\citep{mct, molloylb}  chain for sampling \emph{polygon triangulations} via a \emph{flip walk}, the first and second author developed a framework for bounding the expansion of a graph by recursively decomposing the graph into subgraphs, and analyzing the mutual connectivity of the subgraphs at each level. Prior decomposition techniques existed in the MCMC setting, but these were too lossy to be of use for the triangulation flip walk. The triangulation flip walk has, as its state space, a well-studied graph known as the \emph{associahedron}. Thus the problem of bounding the triangulation flip walk mixing time is equivalent to the problem of (lower-)bounding the expansion of the associahedron.

In a companion paper~\citep{eppfrisharx}, the first and second authors further developed the framework and applied it to a number of other Markov chains. 

\subsection{Our contribution}
In this paper, we apply the previous Markov chain decomposition framework to a large graph family that does not admit a natural Markov chain interpretation, showing that this framework is of independent interest beyond MCMC analysis. Furthermore, the family of graphs we are studying\textemdash Hanoi graphs\textemdash do not admit a direct application of that framework, as written. We further develop the framework and show it can nonetheless be applied to Hanoi graphs. 

The best known treewidth upper bound for~$H_p^n$ when~$p \geq 4$ is $O((p-2)^n)$~\citep{hanoitw}. Combining this known upper bound with a new lower bound, we prove the following:
\begin{restatable}{theorem}{thmhanoiexplb}
\label{thm:hanoiexplb}
The expansion of the Hanoi graph $H_p^n$, for fixed $p \geq 3$, is $\Theta(((p-2)/p)^n)$.
\end{restatable}

In Section~\ref{sec:prelim} we will define expansion and treewidth and observe a relationship between the two (\cref{rmk:twexp}); from this relationship it will be immediate that if the expansion of a graph in a given family of bounded-degree graphs (such as Hanoi graphs) $G = (V, E)$ is $\Omega(f(|V|))$, then the treewidth of the same graph is~$|V|\cdot \Omega(f(|V|))$. From this relationship, Theorem~\ref{thm:hanoiexplb}, and the upper bound on treewidth in~\citep{hanoitw} the following corollary is immediate (here $|V| = p^n$):
\begin{corollary}
\label{cor:hanoitwtight}
The treewidth of the Hanoi graph $H_p^n$, for fixed $p \geq 3$, is $\Theta((p-2)^n)$.
\end{corollary}

\section{Preliminaries}
\label{sec:prelim}
\subsection{Hanoi Graphs}

By $D = \{d_1, \dots, d_n\}$ we denote the set of $n$ discs in the Tower of Hanoi puzzle such that whenever $i > j$ the disc $d_i$ is larger than the disc $d_j$. By $[p] = \{1, \dots, p\}$ we denote the $p$ pegs. A \emph{configuration} in the Tower of Hanoi puzzle is given by a function mapping discs to pegs, i.e. $\tau \colon D \rightarrow [p]$. Equivalently write~$\tau$ as a tuple, $\tau = (\tau(d_1), \tau(d_2), \dots, \tau(d_n))$. Since no disc~$d_i$ can be placed on a smaller disc $d_j$, each~$\tau$ indeed uniquely specifies a puzzle state. By $H_p^n$ we denote the Hanoi graph associated with the Tower of Hanoi puzzle on $n$ vertices and $p$ pegs. The vertices of $H_p^n$ are given by the configurations (the functions $\tau \colon D \rightarrow [p]$). There is an edge between two configurations $\tau_1$ and $\tau_2$ if and only if a player can transition from $\tau_1$ to $\tau_2$ by making exactly one legal move (that is, moving a disc from one peg to another, and placing it on top of a larger disc).

We define a decomposition of a Hanoi graph~$H_p^n$ into $p$ induced Hanoi graphs isomorphic to $H_p^{n-1}$ as in~\cite{hanoitw}. For a configuration $\tau$ let $\tau_n$ denote its $n$th component. That is, $\tau_n$ denotes the peg on which the largest disc is placed in the configuration $\tau$. For $i=1,\dots,p$ the sets $\{ \tau \mid \tau_n = i \}$ partition the vertices of $H_p^n$ and the subgraphs induced by these sets are isomorphic to $H_p^{n-1}$. See \cref{fig:hanoi} for an illustration.

When referring to a Hanoi graph $H \cong H_p^n$, we will denote by $\{H_i\mid 1 \leq i \leq p\}$ the induced Hanoi subgraphs of~$H$ that are  induced by the vertex sets $\{\tau \mid \tau_n = i\}$. These subgraphs are mutually vertex-disjoint and are each isomorphic to $H_p^{n-1}$. We call an edge in $H^n_p$ a \emph{boundary edge} if its endpoints lie in two distinct subgraphs~$H_i,H_j$, and we call a vertex a \emph{boundary vertex} if it is incident to a boundary edge. Let $\edgeh{H_i}{H_j} = \{(u, v) \mid u \in H_i, v \in H_j, (u, v) \in E(H)\}$ be the set of boundary edges between a pair of subgraphs, and let $\bdryh{H_i}{H_j} = \{u \in H_i \mid \exists v \in H_j, (u, v) \in E(H)\}$ be the set of boundary vertices in~$H_i$ with a neighbor in~$H_j$.

We now analyze the structure of these subgraphs and the boundaries between them. First we define what we call a \emph{facet}:
\begin{definition}
\label{def:faceth}
If $H$ is a Hanoi graph isomorphic to $H_p^n$, then given $1 \leq i < j \leq p$, let $\faceth{ij}(H) = \faceth{ji}(H)$ be the set of vertices in $H$ that represent puzzle states having no discs on peg $i$ or peg $j$. Call $\faceth{ij}(H)$ a \emph{facet}.
\end{definition}
Under this definition, the restriction of a facet~$\faceth{ij}(H)$ to a Hanoi subgraph~$H_k$ of~$H$ is itself a facet, namely~$\faceth{ij}(H_k)$. Furthermore, by the symmetry among pegs, for all $1 \leq k,l \leq p$ with $k, l \notin \{i,j\}$, the restrictions~$\faceth{ij}(H_k)$ and~$\faceth{ij}(H_l)$ have the same cardinality. That is:
\begin{remark}
\label{rmk:facetdecomp}
Every nonempty facet~$\mathcal{F} = \faceth{ij}(H)$ within a graph~$H \cong H_p^n$, $n \geq 2$, is a union 
\[\mathcal{F}_{ij}(H) = \bigcup_{k\notin\{i,j\}}\faceth{ij}(H_k)\]
(where $\forall k, H_k\cong H_p^{n-1}$) of facets in~$H_p^{n-1}$ copies. All of these facets have the same size,~$|\mathcal{F}|/(p-2)$.
\end{remark}

We now show that the boundary between each $H_i$ and $H_j$ induces a facet within each of $H_i$ and $H_j$:
\begin{lemma}
\label{lem:facetbdry}
Partition $H_p^n$ into $p$ Hanoi subgraphs as $V(H_p^n) = \bigcup_{k=1}^p V(H_k)$, $H_k \cong H_p^{n-1} \forall k$. Then 
\[\faceth{ij}(H_i) = \bdryh{H_i}{H_j}\] 
for all $i,j$.
\end{lemma}
\begin{proof}
The boundary set $\bdryh{H_i}{H_j}$ is precisely the set of vertices in $H_i$ whose configurations in the $H_p^{n-1}$ graph isomorphic to $H_i$ have no disc on peg $i$ or peg $j$. This is because those are the configurations in which the largest disc in $H_p^n$ can be moved from peg $i$ to peg $j$. This set is precisely $\faceth{ij}(H_i)$, by Definition~\ref{def:faceth}.
\end{proof}

We now have the following nice structure:
\begin{restatable}{lemma}{lemhanoimatch}
\label{lem:hanoimatch}
Every pair among the $p$ Hanoi subgraphs comprising $H_p^n$ has as its set of boundary edges a matching of size $(p-2)^{n-1}$, i.e.
    \[
    |\bdryh{H_i}{H_j}| = |\bdryh{H_j}{H_i}| = |\edgeh{H_i}{H_j}| = (p-2)^{n-1}
    \]
\end{restatable}
(Our previous paper~\cite{hanoitw} gives essentially the same proof of this fact.)

\subsection{Expansion and treewidth}
\label{sec:exptw}
Our aim in this paper is to give asymptotically tight bounds on two key parameters for Hanoi graphs: \emph{expansion} and \emph{treewidth}. A graph with large expansion has no sparse cuts; a graph with large treewidth may have sparse cuts, but in a sense we will describe shortly, cannot be recursively decomposed via sparse cuts. Therefore, a lower bound on expansion induces a lower bound on treewidth, and an upper bound on treewidth induces an upper bound on expansion.

We now make these notions precise: given a graph $G = (V, E)$, define the \emph{expansion}, or \emph{edge expansion} $h(G)$ as
\[
  h(G) = \min_{S \subseteq V: |S| \leq |V|/2}\{|\partial S|/|S|\}
\]
where $\partial S = \{(u, v) \mid u \in S, v \in V \setminus S, (u, v) \in E\}$.

Define the \emph{vertex expansion} $h_v(G)$ as
\[
  h_v(G) = \min_{S \subseteq V: |S| \leq |V|/2}\{|\partial_v S|/|S|\}
\]
where $\partial_v S = \{v \in V \setminus S \mid \exists u \in S, (u, v) \in E\}$.

Unless specified as vertex expansion, we mean by ``expansion'' the edge expansion.

Treewidth, while closely related, is defined with respect to a \emph{tree decomposition}: given $G = (V, E)$, let a \emph{tree decomposition} of~$G$, denoted $\mathcal{T} = (\mathcal{X}, \mathcal{E})$ be a tree, whose nodes are a collection of \emph{bags} (a term referring simply to subsets of~$V$) $\mathcal{X} = \{X_1, X_2, \dots, X_k\}$, satisfying the following conditions:
\begin{enumerate}
    \item For all $v \in V$, some bag $X_i$ contains $v$.
    \item For all $(u, v) \in E$, some bag~$X_i$ contains $u$ and~$v$.
    \item For all $v \in V,$ the set of bags containing~$v$ induces a nonempty (and connected) subtree of~$\mathcal{T}$.
\end{enumerate}
Define the \emph{width} of a tree decomposition~$\mathcal{T} = (\mathcal{X}, \mathcal{E})$ to be one less than the cardinality of the largest bag in~$\mathcal{X}$. Define the \emph{treewidth} of a graph~$G$ to be the smallest possible width of a tree decomposition of~$G$. Denote the treewidth by $\tw(G)$.

Define a \emph{balanced vertex separator} $X \subseteq V$ in a graph $G = (V, E)$ as a set of vertices such that~$V \setminus X$ can be partitioned into two subsets~$A$ and~$B$ with $|V|/3 \leq |A| \leq |B| \leq 2|V|/3$, where no vertex in~$A$ has a neighbor in~$B$.

\begin{lemma}\citep{ericksontw}
\label{lem:balsep}
Given a graph~$G$, let~$t = \tw(G)$. Then~$G$ has a balanced vertex separator of size at most~$t+1$.
\end{lemma}

\cref{lem:balsep} immediately implies the following, by the definition of a balanced separator and of vertex expansion:
\begin{corollary}
\label{cor:sepvexp}
Let $G = (V, E)$ be a graph and let $t = \tw(G)$. Then the vertex expansion of~$G$ is at most~$\frac{3(t+1)}{|V|}$.
\end{corollary}

We observe the following consequence of the definitions of treewidth, expansion, and vertex expansion and of \cref{cor:sepvexp}:
\begin{remark}
\label{rmk:twexp}
    In a graph~$G$ of degree~$\Delta$, letting $t = \tw(G)$, we have 
    \[h_v(G) \leq h(G) \leq \Delta\cdot h_v(G) \leq \frac{3\Delta(t+1)}{|V|}\]
\end{remark}

\subsection{Expansion and multicommodity flows}
\label{sec:expmcflow}
\emph{Multicommodity flows} are a standard tool for lower-bounding the expansion of a graph~\citep{sinclair_1992, kaibelexp}. Let $G = (V, E)$ be a graph. Replace each (undirected) edge $\{u, v)\} \in E$ by two directed edges $(u, v)$ and $(v, u)$. Denote the resulting set of directed edges by~$A$. A multicommodity flow~$f$ is a collection of functions~$\{f_{st}: A \rightarrow \mathbb{R}_{\geq 0} \mid (s, t) \in V \times V\}$ obeying the standard flow axioms:
\begin{enumerate}
    \item $\sum_{v \in N(s)} f_{st}(s, v) = \sum_{v \in N(t)} f_{st}(v, t) = 1$~\footnote{In this paper we only consider \emph{uniform} multicommodity flows, in which the right-hand side of the equality in the first condition is~$1$.}
    \item for all $v \in V \setminus \{s, t\}$, $\sum_{u \in N(v)} f_{st}(u, v) = \sum_{w \in N(v)} f_{st}(v, w)$
\end{enumerate}
By $N(v)$ we mean $\{u \in V: (u, v) \in E\}$.

Define the \emph{congestion} of a multicommodity flow~$f$ as $\frac{1}{|V|}\max_{(u, v) \in A}\sum_{(s, t) \in V \times V}f_{st}(u, v)$.

The following standard inequality relates multicommodity flows to expansion (see e.g.~\cite{sinclair_1992}):
\begin{lemma}
\label{lem:flowexp}
If a graph $G = (V, E)$ admits a multicommodity flow~$f$ with congestion at most~$\rho$, then
\[h(G) \geq 1/(2\rho)\]
\end{lemma}

\section{Extending a prior framework}
\label{sec:framework}
In two prior works~\citep{eppfrishtri, eppfrisharx} the first and second author presented a framework consisting of a set of conditions that guarantee \emph{rapid mixing} for the natural random walk on a family of graphs on $N = f(n)$ vertices, where the family is parameterized by $n << N$. This condition is equivalent to the graph having expansion $\Omega\left(1/n^{O(1)}\right)$. They applied the framework to obtain rapid mixing results for random walks on the \emph{triangulations} of a convex point set, as well as for a number of graph-theoretic sampling problems in bounded-treewidth graphs.

As we discussed in that paper, the framework is an analogue of the \emph{projection-restriction} technique of Jerrum, Son, Tetali, and Vigoda~\citep{jerrumprojres}. The idea is to recursively partition the vertices of the graph into a small number of induced subgraphs in the same family, allowing a recursive decomposition. One can then construct a multicommodity flow recursively, by first specifying the amount of flow to send across each edge between a pair of the subgraphs, and then solving the resulting flow subproblems recursively.

The framework in~\cite{eppfrishtri} is as follows:
\begin{lemma}
\label{lem:eppfrishfw}
Let~$\mathcal{F} = \{\mathcal{M}_1, \mathcal{M}_2, \dots\}$ be an infinite family of connected graphs, parameterized by a value $n$. Suppose that for every graph~$\mathcal{M}_n = (\mathcal{V}_n, \mathcal{E}_n) \in \mathcal{F}$, for $n \geq 2$, the vertex set~$\mathcal{V}_n$ can be partitioned into a set~$\mathcal{S}_n$ of classes inducing subgraphs of~$\mathcal{M}_n$ that satisfy the following conditions:
\begin{enumerate}
\item\label{addcondcart} Each subgraph is isomorphic to a smaller graph $\mathcal{M}_{i} \in \mathcal{F}$, $i < n$.
\item\label{addcondnum} The number of classes is at most~$n^{O(1)}$.
\item\label{addcondmatch} For every pair of classes~$\mathcal{C}, \mathcal{C'}\in \mathcal{S}_n$, the set of edges between the subgraphs induced by the two classes is a matching of size at least~$\frac{|\mathcal{C}||\mathcal{C'}|}{|\mathcal{V}_n|}.$
\item\label{addcondbdry} Given a pair of classes~$\mathcal{C}, \mathcal{C'} \in \mathcal{S}_n$, the set of vertices in $\mathcal{C}$ having a neighbor in $\mathcal{C'}$ induces a subgraph of $\mathcal{C}$ that is isomorphic to a smaller graph $\mathcal{M}_{j} \in \mathcal{F}$, $j < n$.
\end{enumerate}

Suppose further that~$|\mathcal{V}_1| = 1$. Then the expansion of~$\mathcal{M}_n$ is~$\Omega(1/(\kappa(n)n))$, where~$\kappa(n) = \max_{1\leq i\leq n}|\mathcal{S}_i|$ is the maximum number of classes in any~$\mathcal{M}_i, i\leq n$.
\end{lemma}

In fact, Eppstein and Frishberg~\citep{eppfrishtri} proved a stronger version of \cref{lem:eppfrishfw} that holds if the subgraphs in the decomposition are isomorphic to \emph{Cartesian products} of smaller graphs in the family. However, these Cartesian products are not relevant to the case of Hanoi graphs.

In our case, the induced subgraphs are straightforward to find: in an $H_p^n$ input graph, we define the subgraphs to be the $p$ copies of $H_p^{n-1}$ that correspond to playing the puzzle on the smallest $n - 1$ discs given each of the $p$ possible placements of the largest disc.

\cref{lem:eppfrishfw} requires a strong lower bound (condition~\ref{addcondmatch}) on the number of edges between a given pair of the subgraphs, namely that the number of edges must be at least equal to the product of the cardinalities of the two subgraphs, divided by the total number of vertices in the graph. This condition fails for Hanoi graphs: this ratio is inverse exponential in~$n$. Nonetheless, we show how to repair the framework and apply it to Hanoi graphs. Informally, we use the fact that although the number of edges is much smaller than the stated bound, the \emph{boundary set} between two subgraphs\textemdash the set of vertices in one subgraph $H_p^{n-1}$ having neighbors in a given other subgraph\textemdash is well distributed among the $H_p^{n-2}$ subgraphs that comprise $H_p^{n-1}$. We require also that each subgraph admits a decomposition into ``facets'' (\cref{sec:prelim}) allowing the construction of many ``parallel'' paths from such a boundary set to the other vertices within the subgraph. This avoids too much concentration of flow within a given vertex.

On the other hand, the decomposition of the Hanoi graph that we give here is simpler than those in~\cite{eppfrishtri}: that framework allowed each subgraph to be a \emph{Cartesian product} of smaller graphs; here we have simply a smaller similar graph\textemdash and that graph is in fact completely determined by its isomorphism to $H_p^{n-1}$.

Thus the main contribution of this paper, in addition to establishing tight expansion and treewidth bounds for $H_p^n$, is to give an extension of the framework where some of the conditions fail.

\begin{figure}
\includegraphics[height=12em]{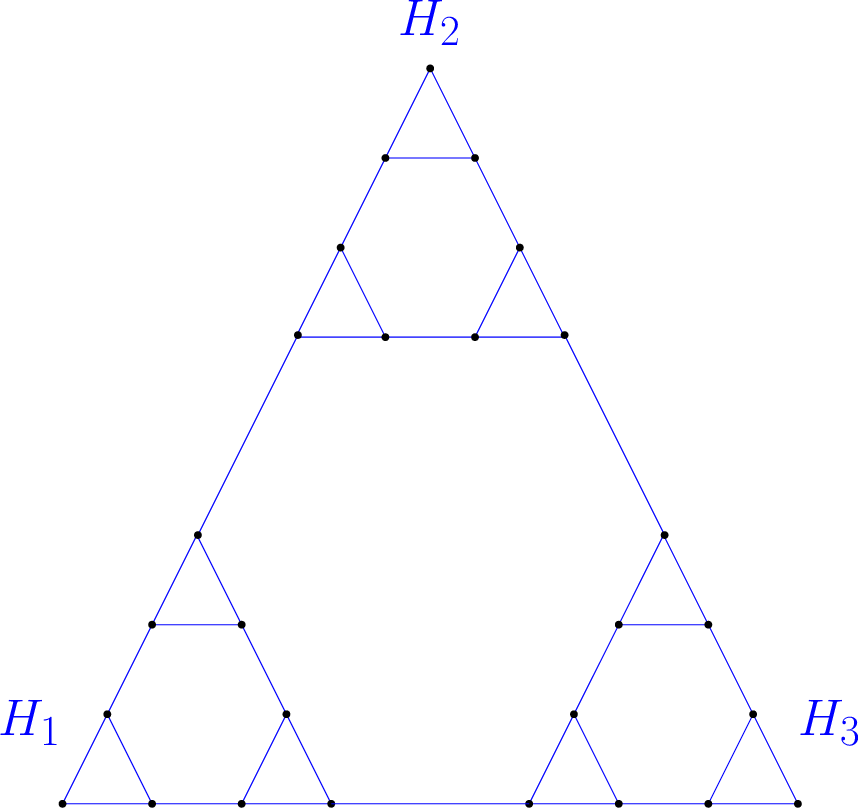}
\hspace*{1.5em}
\includegraphics[height=12em]{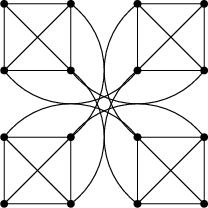}
\hspace*{1.5em}
\includegraphics[height=12em]{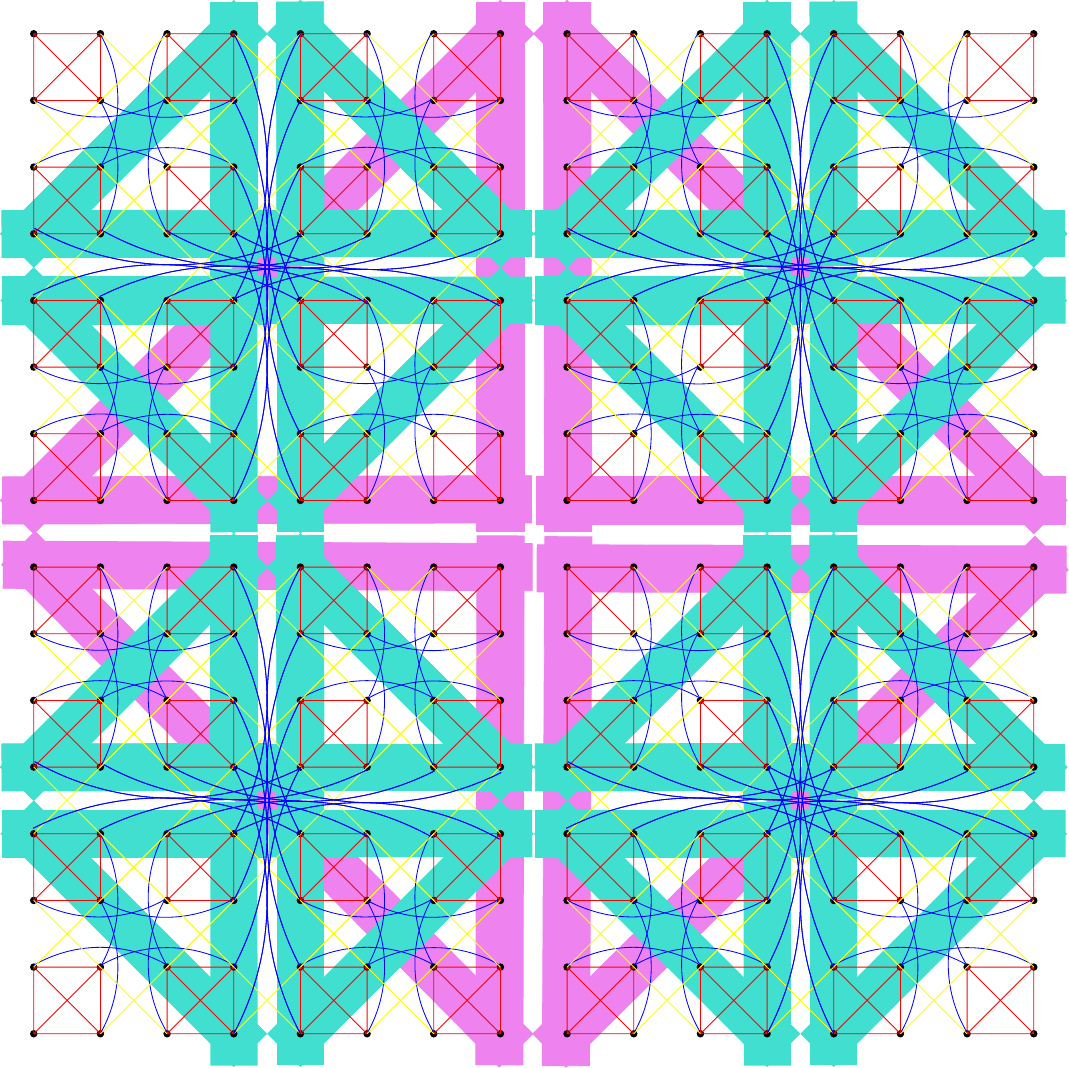}
\hspace*{1.5em}
\caption{The Hanoi graphs $H_3^3$ (left), $H_4^2$ (center), and $H_4^4$ (right), with their vertices arranged recursively in subgraphs. At the top level, the three (or four) subgraphs in each graph are induced by placements of the largest disc in the puzzle. In $H_4^4$, we depict in different colors the boundary matchings at different recursive levels of the graph. At the top two levels, we omit individual edges, and instead highlight the boundary sets in magenta (top level) and turquoise (second level).}
\label{fig:hanoi}
\end{figure}

\section{Large expansion for Hanoi graphs}
\label{sec:hanoiexplb}
In this section we adapt the framework of~\cite{eppfrishtri} to Hanoi graphs. We will prove the following:

\begin{restatable}{lemma}{lemhanoiindflow}
\label{lem:hanoiindflow}
Let $p \geq 3$ be fixed. Suppose there exists a multicommodity flow~$f$ in $H_p^{n-1}$ with congestion $\rho$. Then there exists a multicommodity flow in $H_p^n$ with congestion $\rho + O((p/(p-2))^n)$.
\end{restatable}

Then Theorem~\ref{thm:hanoiexplb} will immediately follow via induction on $n$ (with a trivial base case).

The construction in this section largely follows the construction used in~\cite{eppfrishtri} to prove \cref{lem:eppfrishfw}. However, applying their construction breaks down in our case due to the small sets of edges between Hanoi subgraphs. We address this with the new machinery in \cref{sec:handetails}. 

The inductive structure comes from dividing $H_p^n$ into $p$ subgraphs as described in Section~\ref{sec:framework}, each of which is isomorphic to $H_p^{n-1}$. We then reuse the fractional paths given by the flow $f$ for pairs of vertices lying in the same $H_p^{n-1}$ copy. The rest of this section is devoted to routing flow between those pairs lying in different $H_p^{n-1}$ copies. 

Let $H_1, \dots, H_p$ be the $p$ copies of $H_p^{n-1}$. Suppose a vertex $s \in H_1$ wants to send an equal amount of flow to all vertices $t \in H_2$. We first find some set of paths along which $s$ sends its flow to the \emph{boundary vertices} in $H_1$\textemdash those having neighbors in $H_2$\textemdash so that the flow is equally distributed among those boundary vertices. We let those boundary vertices simply send the flow across their edges to $H_2$; then the boundary vertices in $H_2$ need to find a way to distribute the flow to the rest of the vertices in $H_2$.

We call this last step \emph{distribution}; we call sending the flow across the boundary the \emph{transmission} step; we will divide the first part\textemdash having $s$ send its flow to the boundary\textemdash into two steps, one of which we call \emph{shuffling} and one of which we call \emph{concentration}. The shuffling step, which we will describe shortly, takes advantage of the inductive structure we described above.

Figure~\ref{fig:trandist} shows the transmission and distribution steps in $H_3^3$. (The distribution step is almost trivial in this case, since the boundary matching has size one.)

\begin{figure}
\label{fig:trandist}
\includegraphics[height=20em]{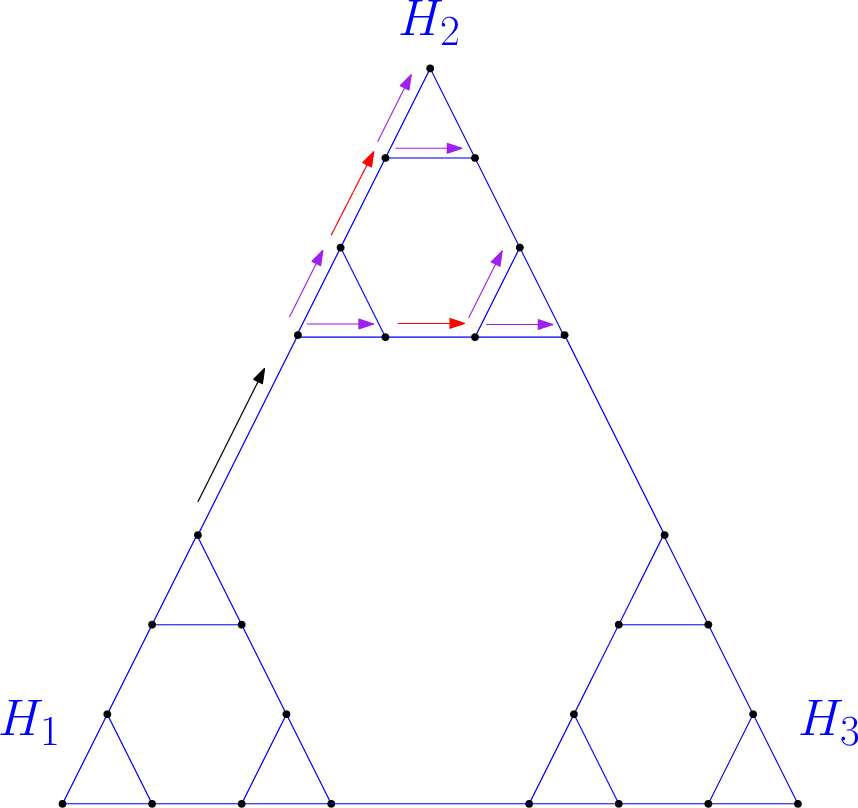}
\caption{The transmission and distribution steps for $H_3^3$ as described in Section~\ref{sec:hanoiexplb}. 
The top vertex in~$H_1$ routes $|V(H_1)||V(H_2)|$ flow to its neighbor in~$H_2$. That neighbor, the bottom-left vertex in $H_2$, then routes $|V(H_1)|\cdot 3$ flow to its topmost neighbor (similarly right neighbor) that is bound for the topmost (similarly bottom-right) 3-clique in $H_2$, and distributes an additional $|V(H_1)|\cdot 1$ flow to each neighbor in its own 3-clique.}
\end{figure}

To formalize the division of this multicommodity flow into ``steps,'' we define the steps as a more general type of flow, as in~\cite{eppfrishtri}:

Given a graph $G = (V, E)$, denote by~$A$ the set of directed arcs obtained by directing the edges of~$E$ in both directions as in the definition of a multicommodity flow in \cref{sec:expmcflow}.
Define a \emph{multi-way single-commodity flow} (MSF)~\citep{eppfrishtri} as a function~$f: A \rightarrow \mathbb{R}_{\geq 0}$ with \emph{source} set $S \subseteq V$, \emph{sink} set $T \subseteq V$, and a set of \emph{surplus} and \emph{demand} amounts $\sigma: S \rightarrow \mathbb{R}$ and $\delta: T \rightarrow \mathbb{R}$, as a flow $f: A \rightarrow \mathbb{R}$ in $G$, such that:
\begin{enumerate}
\item for all $s \in S \setminus T$ we have $\sum_{v \in N(s)} (f(s, v) - f(v, s)) = \sigma(s)$
\item for all $t \in T \setminus S$ we have $\sum_{v \in N(t)} (f(v, t) - f(t, v)) = \delta(t)$ 
\item for all $u \in S \cap T$ we have $\sum_{v \in N(u)} (f(u, v) - f(v, u)) = \sigma(u) - \delta(u)$
\item for all $u \in V \setminus (S \cup T)$ we have $\sum_{v \in N(u)}(f(u, v) - f(v, u)) = 0$
\end{enumerate}

Define by an \emph{MSF problem} a tuple $\pi = (S, T, \sigma, \delta)$,uio where~$S,T,\sigma,\delta$ are as described. Say that an MSF~$f$ \emph{solves} the MSF problem $\pi = (S, T, \sigma, \delta)$ if~$f$ is consistent with $S,T,\sigma,\delta$.

In some cases $\sigma$ or $\delta$ will be a constant function, and in those cases we will abuse notation and write $\sigma$ (resp. $\delta$) for the value $\sigma(s)$ (resp. $\delta(t)$).

It follows from the definition of an MSF that, given MSFs $f_1, f_2$ solving $\pi_1 = (S_1, T_1, \sigma_1, \delta_1)$ and $\pi_2 = (S_2 = T_1, T_2, \sigma_2 = \delta_1, \delta_2)$  one can \emph{compose} the two MSFs $f_1, f_2$\textemdash combine them into a third MSF $f_2 \circ f_1$ that solves what we will call the composition 
\[
\pi_2 \circ \pi_1 = (S_1, T_2, \sigma_1, \delta_2).\]

We can also take the \emph{sum} of~$f_1$ and $f_2$ to be $f_3 = f_1 + f_2$, defined pointwise, so that $f_3$ solves what we will call the sum of $\pi_1$ and $\pi_2$, i.e.
\[\pi_1 + \pi_2 = (S_1 \cup S_2, T_1 \cup T_2, \sigma_3, \delta_3);
\]
for $\sigma_3$ and $\delta_3$ we naturally extend the domain to $S_1 \cup S_2$:
first, extend $\sigma_1, \sigma_2, \delta_1, \delta_2$: for $v \in S_2 \setminus S_1,$ let $\sigma_1(v) = 0$; similarly for $v \in S_1 \setminus S_2,$ let $\sigma_2(v) = 0$. For $v \in T_2 \setminus T_1$ let $\delta_1(v) = 0$; similarly, for $v \in T_1 \setminus T_1$ let $\delta_2(v) = 0$.

Finally for all $v$ let $\sigma_3(v) = \sigma_1(v) + \sigma_2(v)$ and let $\delta_3(v) = \delta_1(v) + \delta_2(v)$.

Now consider any two of the $H_p^{n-1}$ subgraphs, without loss of generality~$H_1$ and~$H_2$. Consider the aggregate problem of letting each vertex $s \in V(H_1)$ send a separate unit to every $t \in V(H_2)$. We will decompose this problem into a collection of MSF subproblems, one for each source vertex~$s$. For each~$s$, call this subproblem

\[
    \pi_s = (\{s\}, V(H_2), \sigma_s = |V(H_2)|, \delta_s = 1).
\]

\begin{remark}
    To send flow from every $s \in V(H_1)$ to every $t \in V(H_2)$, it suffices to define a function $f_s$ for each $s \in V(H_1)$, such that $f_s$ solves $\pi_s$.

    Furthermore, the resulting congestion is equal to the congestion produced by the function $\sum_s f_s$.
\end{remark}

Now we describe and formalize the shuffling step we have previously described\textemdash similar to the shuffling step in~\cite{eppfrishtri}\textemdash as an MSF problem~$\pi_{shuf}$. We also define \emph{concentration}, \emph{transmission}, and \emph{distribution} problems $\pi_{conc}, \pi_{tran},$ and~$\pi_{dist}$.

\begin{lemma}
    \label{lem:sproblems}
    Given a vertex $s \in V(H_1)$, define the following MSF problems:
    \begin{itemize}
        \item $\pi_{shuf} = \left(\{s\}, V(H_1), \sigma_{shuf} = |V(H_2)|, \delta_{shuf} = \frac{|V(H_2)|}{|V(H_1)|} = 1\right),$
        \item $\pi_{conc} = \left(V(H_1), \bdryh{H_1}{H_2}, \sigma_{conc} = 1, \delta_{conc} = \frac{|V(H_1)|}{|\bdryh{H_1}{H_2}|}\right),$
        \item $\pi_{tran} = \left(\bdryh{H_1}{H_2}, \bdryh{H_2}{H_1}, \sigma_{tran} = \delta_{tran} = \frac{|V(H_1)|}{|\bdryh{H_1}{H_2}|} = \frac{|V(H_2)|}{|\bdryh{H_2}{H_1}|}\right)$
        \item $\pi_{dist} = \left(\bdryh{H_2}{H_1}, V(H_2), \sigma_{dist} = \frac{|V(H_2)|}{|\bdryh{H_2}{H_1}|}, \delta_{dist} = 1\right)$
    \end{itemize}
    The composition of these problems is~$\pi_s = \pi_{dist} \circ \pi_{tran} \circ \pi_{conc} \circ \pi_{shuf}$.
\end{lemma}
\begin{proof}
    The claim follows from comparing source and sink sets, as well as the~$\sigma$ and~$\delta$ functions, of the sequence of MSF problems ~$\pi_{shuf}$, $\pi_{conc}$, $\pi_{tran}$, and $\pi_{dist}$.
\end{proof}

The idea of the shuffling step~$\pi_{shuf}$ is, given a vertex $s \in V(H_1)$ that needs to send flow to every $t \in V(H_2)$, we let $s$ first send an equal amount (a $1/|V(H_1)|$ factor) of this flow to every $s' \in V(H_1)$, using the flow $f$ that exists in $V(H_1)$ by the inductive hypothesis in Lemma~\ref{lem:hanoiindflow}.

The idea of~$\pi_{conc}$, the \emph{concentration} MSF, is\textemdash after we have uniformly distributed the outbound flow from~$s$ throughout~$V(H_1)$\textemdash to then concentrate this flow on the boundary~$\bdryh{H_1}{H_2}$.

To send flow from~$\bdryh{H_1}{H_2}$ across the edges~$\edgeh{H_1}{H_2}$ between $H_1$ and $H_2$, we need to solve the \emph{transmission} MSF,
$\pi_{tran}.$

Finally, to ensure that all of the flow is distributed evenly throughout~$V(H_2)$ we need to solve the \emph{distribution} step~$\pi_{dist}$.

\cref{lem:sproblems} states that to solve~$\pi_s$ it suffices to solve $\pi_{shuf}, \pi_{conc}, \pi_{tran},$ and $\pi_{dist}$.

We now show how to solve~$\pi_{shuf}$:

\begin{lemma}
\label{lem:pishufcong}
Suppose there exists a multicommodity flow~$f$ in $H_p^{n-1}$ with congestion~$\rho$. Then for each $s \in V(H_1)$ the MSF problem~$\pi_{shuf}$ defined with respect to vertex~$s$ can be solved with an MSF $f_{shuf}$, such that summing over all such MSFs produces congestion at most $\rho\cdot \frac{|V(H_1)|}{|V\left(H_p^n\right)|}$.
\end{lemma}
\begin{proof}
To solve all of the $\pi_{shuf}$ problems (one for each $s \in V(H_1)$, every $s \in V(H_1)$ must send a single unit ($|V(H_2)|/|V(H_1)| = 1$) of flow to every $s' \in V(H_1)$. In other words, this is a \emph{multicommodity} flow in $H_1$. By the assumption that~$f$ exists in $H_p^{n-1}$ with congestion~$\rho$ and the fact that $H_1 \cong H_p^{n-1}$, this is accomplished with congestion
\[\rho\cdot \frac{|V(H_1)|}{|V(H_p^n)|}\]
by $f$: here we are passing from a flow defined over $H_1$ to a flow defined over $H_p^n$; since the congestion is defined to be normalized by the number of vertices in the graph $H_p^n$ (\cref{sec:expmcflow}), we have passed from normalizing by $\frac{1}{|V(H_1)|}$ to normalizing by $\frac{1}{|V(H_p^n)|}$, giving the scaling factor $\frac{|V(H_1)|}{|V(H_p^n)|}$.

Let $f_{shuf}$ be identical to this function $f$ over $H_1$, and zero everywhere else.
\end{proof}

Next we solve $\pi_{tran}$:
\begin{lemma}
    The MSF problem $\pi_{tran}$ for each vertex $s\in V(H_1)$ can be solved, such that the total congestion for all such flows is at most $\frac{|V(H_1)||V(H_2)|}{|\edgeh{H_1}{H_2}||V(H_p^n)|}$.
\end{lemma}
\begin{proof}
Let $f_{tran}$ simply send an equal amount $\sigma_{tran} = \delta_{tran}$ of flow across every boundary edge in the matching $\edgeh{H_1}{H_2}$. Summing over all $|V(H_1)|$ such flows and normalizing by $|V(H_p^n)|$ gives congestion
\[
    \frac{|V(H_1)|\sigma_{tran}}{|V(H_p^n)|} = \frac{|V(H_1)|\delta_{tran}}{|V(H_p^n)|} = \frac{|V(H_1)||V(H_2)|}{|\edgeh{H_1}{H_2}||V(H_p^n)|}.
\]
\end{proof}

The hard part is to solve $\pi_{conc}$ and $\pi_{dist}$. Fortunately, however, these two MSFs are ``mirror images'' of one another: swapping the identities of $H_1$ and $H_2$, of sink and source vertices, and of $\sigma$ and $\delta$ functions is sufficient to transform one into the other. By this symmetry it suffices to solve $\pi_{dist}$ and bound its congestion.

We do so in Section~\ref{sec:handetails}.

\section{The distribution step}
\label{sec:handetails}
In this section we introduce new machinery that allows us to repair the framework in~\cite{eppfrishtri}. As we stated in \cref{sec:hanoiexplb}, that framework breaks down for Hanoi graphs. In particular, Condition 3 of \cref{lem:eppfrishfw} fails, as the number of edges between a pair of Hanoi subgraphs is much smaller than needed for that condition.

Condition 3 requires that the number of edges between a pair of subgraphs be at least equal to the product of the cardinalities of the vertex sets of the subgraphs, divided by the total number of edges in the graph. This would give congestion at most one in the $f_{tran}$ flow defined in \cref{sec:hanoiexplb}.

However, since Condition 3 fails for Hanoi graphs, the $f_{tran}$ flow has congestion (much) greater than one, causing an exponential blowup in the recursive construction that is needed to solve the $\pi_{conc}$ and $\pi_{dist}$ subproblems.

We address this problem by more carefully specifying the flow $f_{dist}$ that solves the $\pi_{dist}$ subproblem. To do so, we use the facets we defined in \cref{sec:prelim}.

Recall that we defined 
\[\pi_{dist} = \left(\bdryh{H_2}{H_1}, V(H_2), \sigma_{dist} = \frac{|V(H_2)|}{|\edgeh{H_1}{H_2}|}, \delta_{dist} = 1\right)\]
as the problem of distributing $|V(H_2)|$ units of flow, sent from some vertex $s \in H_1$, equally throughout $H_2$. 
Our goal is to construct $f_{dist}$ so that it produces congestion at most $\sigma_{dist}$. This will imply that the sum of $|V(H_1)|$ such flows\textemdash one for each $s \in V(H_1)$\textemdash produces at most 
\[|V(H_1)|\cdot\sigma_{dist} = |V(H_1)| \cdot \frac{|V(H_2)|}{|\edgeh{H_1}{H_2}|} = (p/(p-2))^{n-2} \cdot |V(H_p^n)|\] units of flow across edges in $H_2$, yielding $(p/(p-2))^{n-2}$ (normalized) congestion in the inductive step in Lemma~\ref{lem:hanoiindflow}, as desired.

Define $\{H_{1,i} \mid 1 \leq i \leq p\}$ as the $p$ Hanoi subgraphs of~$H_1$. Define~$\{H_{2,i}\}$ similarly. 
We construct~$f_{dist}$ recursively by decomposing~$\pi_{dist}$ into MSF subproblems:
\begin{lemma}
\label{lem:fdistrecdef}
Define the following MSF problems:
\begin{itemize}
    \item For $3 \leq i \leq p$ define the ``distribution'' subproblem \begin{align*}
        \pi_{distsub,i} = \left(\bdryh{H_2}{H_1} \cap V(H_{2,i}), V(H_{2,i}), \sigma_{distsub,i} = ((p-2)/p)\sigma_{dist}, \delta_{distsub,i} = 1\right)
    \end{align*}

    \item Define the ``distribution'' subproblems
    \begin{align*}
        \pi_{distsub,1} &= \left( \bigcup_{3 \leq i\leq p}\bdryh{H_{2,1}}{H_{2,i}}, V(H_{2,1}), \sigma_{distsub,1} = (1/p)\sigma_{dist}, \delta_{distsub,1} = 1\right) \\
        \pi_{distsub,2} &= \left(\bigcup_{3 \leq i\leq p}\bdryh{H_{2,2}}{H_{2,i}}, V(H_{2,2}), \sigma_{distsub,2} = (1/p)\sigma_{dist}, \delta_{distsub,2} = 1\right)
    \end{align*}

    \item For $3 \leq i \leq p$ define the ``transmission'' subproblems
    \begin{align*}
        \pi_{tran,i,1} = (\bdryh{H_{2,i}}{H_{2,1}}, \bdryh{H_{2,1}}{H_{2,i}}, \sigma_{tran,i,1} = \delta_{tran,i,1} = (1/p)\sigma_{dist}) \\
        \pi_{tran,i,2} = (\bdryh{H_{2,i}}{H_{2,2}}, \bdryh{H_{2,2}}{H_{2,i}}, \sigma_{tran,i,2} = \delta_{tran,i,2} = (1/p)\sigma_{dist})
    \end{align*}
    
    \item Define the ``routing'' subproblems
    \begin{align*}
        \pi_{rout,i,1} = (\bdryh{H_2}{H_{1}} \cap V(H_{2,i}), \bdryh{H_{2,i}}{H_{2,1}}, \sigma_{rout,i,1} = \delta_{rout,i,1} = \sigma_{tran,i,1}) \\
        \pi_{rout,i,2} = (\bdryh{H_2}{H_{1}} \cap V(H_{2,i}), \bdryh{H_{2,i}}{H_{2,2}}, \sigma_{rout,i,2} = \delta_{rout,i,2} = \sigma_{tran,i,2})
    \end{align*}

\end{itemize}
Any collection of flows that solves these problems solves~$\pi_{dist}$. More precisely,
\begin{align*}
    \pi_{dist} &= \sum_{3\leq i \leq p} \pi_{distsub,i} \\
    &\quad + \pi_{distsub,1} \circ \sum_{3\leq i \leq p} \pi_{tran,i,1} \circ \pi_{rout,i,1} \\
    &\quad + \pi_{distsub,2} \circ \sum_{3\leq i \leq p} \pi_{tran,i,2} \circ \pi_{rout,i,2}
\end{align*}
\end{lemma}

\begin{figure}
    \centering
    \includegraphics[width=0.4\linewidth]{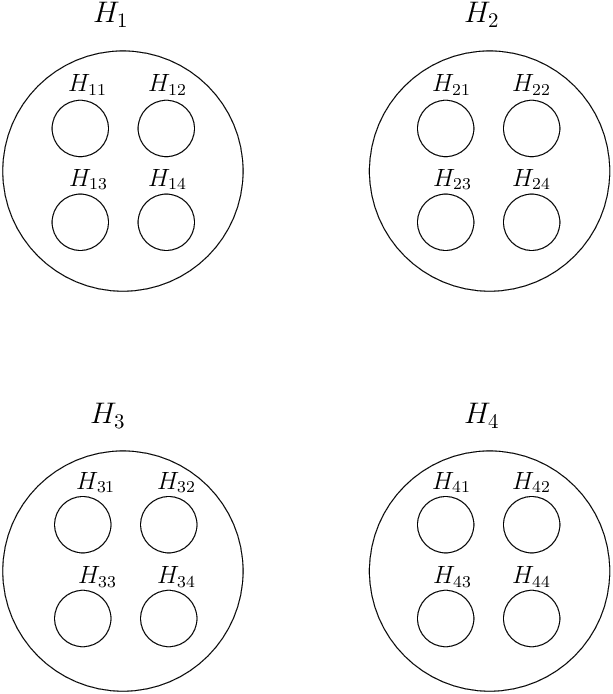}
    \caption{A schematic illustration showing the decomposition of the Hanoi subgraphs into subgraphs. $H_{ij}$ is the subgraph induced by the set of configurations in which the largest disc is placed on peg~$i$ and the second-largest disc is placed on peg~$j$.}
    \label{fig:decomp}
\end{figure}

\begin{figure}
    \centering
    \includegraphics[width=0.45\linewidth]{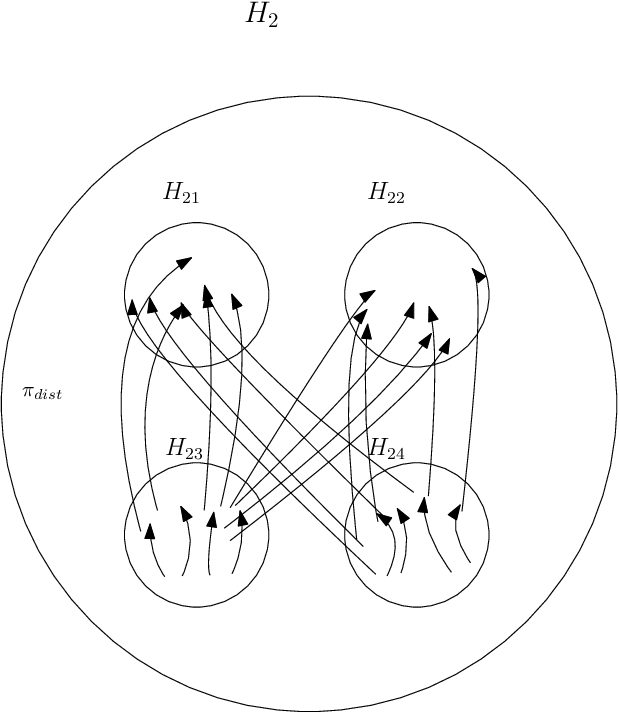}
    \hspace{0.4em}
    \includegraphics[width=0.45\linewidth]{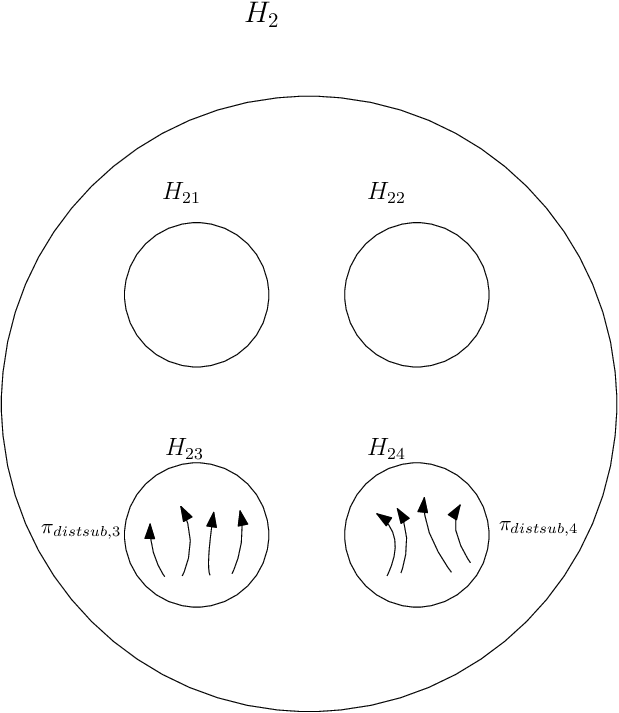}
    
    \vspace{2em}
    \includegraphics[width=0.45\linewidth]{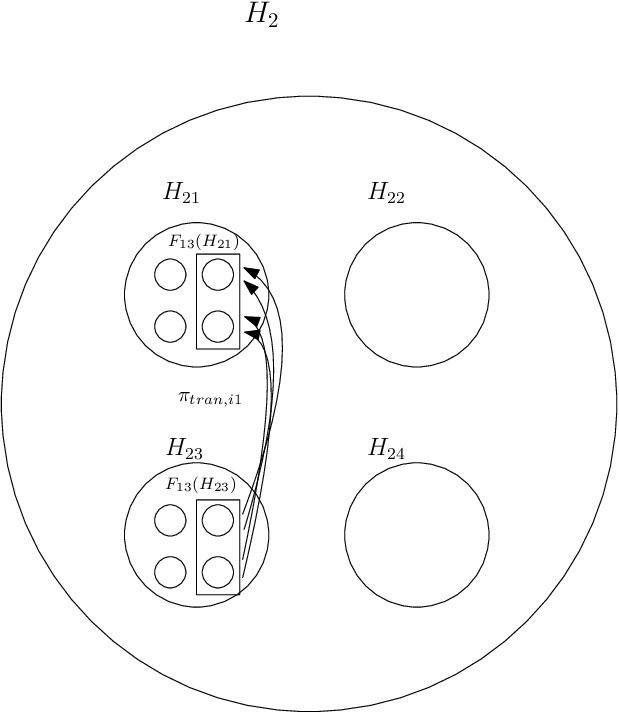}

    \caption{Top left: The problem~$\pi_{dist}$ is the problem of distributing flow received from $H_1$ throughout~$H_2$. Since $\bdryh{H_2}{H_1} = \faceth{12}(H_2) = \faceth{12}(H_{23}) \cup \faceth{12}(H_{24})$, this flow is initially concentrated in the facets $\faceth{12}(H_{23})$ and $\faceth{12}(H_{24})$ and must be distributed throughout~$H_{2}$.\\
    Top right: The problems~$\pi_{distsub,3}$ and $\pi_{distsub,4}$ of distributing flow received from $H_1$ throughout $H_{23}$ and $H_{24}$. This flow is initially concentrated within $\faceth{12}(H_{23})$ and $\faceth{12}(H_{24})$.\\
    Bottom: $\pi_{tran,31}$ is a subproblem of~$\pi_{distsub,3}$ and involves sending the $H_{23}\rightarrow H_{21}$ flow (which originates at~$H_1$) across the boundary matching~$\edgeh{H_{23}}{H_{21}}$.}
    \label{fig:pidistrec}
\end{figure}

\begin{figure}
    \centering
    \includegraphics[width=0.4\linewidth]{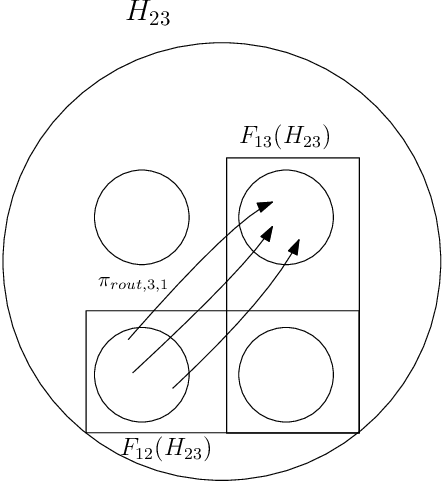}
    \includegraphics[width=0.4\linewidth]{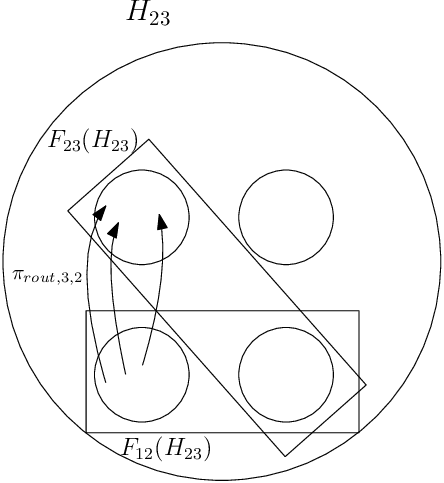}
    \caption{An illustration of the ``routing'' problems $\pi_{rout,i,1}, \pi_{rout,i,2}$ in \cref{lem:fdistrecdef}.}
    \label{fig:pirout}
\end{figure}

We briefly describe each of the problems defined in \cref{lem:fdistrecdef}. (Figure~\ref{fig:decomp}, Figure~\ref{fig:pidistrec}, and Figure~\ref{fig:pirout} visualize these problems for the case $p = 4$.) Our task in constructing $f_{dist}$ is to distribute flow from the boundary set~$\bdryh{H_2}{H_1}$ in~$H_2$ throughout the rest of $H_2$. Consider that
\[\bdryh{H_2}{H_1}~=~\faceth{12}(H_2)~=~\bigcup_{i=3}^p\faceth{12}(H_{2i}) = \bigcup_{i=3}^p\left(\bdryh{H_2}{H_1} \cap V(H_{2,i})\right)\] is both itself a facet, and also a union of facets within Hanoi subgraphs of $H_2$. Thus~$\pi_{distsub,i}$ is the problem of distributing the flow received from $s \in H_1$ evenly throughout the~$i$th subgraph $V(H_{2,i})$.

For $3 \leq i \leq p$, $H_{2,i} \cap \bdryh{H_2}{H_1} = \faceth{12}(H_{2,i}) \neq \emptyset$; therefore $H_{2,i}$ receives flow from $H_1$ directly. This flow must be distributed evenly throughout $H_{2,i}$; this is accomplished by solving $\pi_{distsub,i}$.

However, $V(H_{2,1})$ and $V(H_{2,2})$ do not contain boundary vertices, i.e. $V(H_{2,1}) \cap \bdryh{H_2}{H_1} = \emptyset$ and $V(H_{2,2}) \cap \bdryh{H_2}{H_1} = \emptyset$. Therefore, we must route flow from the other $H_{2,i}$ ($3 \leq i \leq p$) subgraphs to $V(H_{2,1})$ and $V(H_{2,2})$ before solving $\pi_{distsub,1}$ and $\pi_{distsub,2}$. The problem of doing so is to solve $\pi_{tran,i,1}$ and $\pi_{rout,i,1}$ ($1 \leq 3 \leq p$).

More precisely,~$\pi_{tran,ij}$ is the problem of transmitting enough flow across the boundary~$\edgeh{H_{2i}}{H_{2j}}$ to ensure that all of the~$p$ subgraphs of $H_2$ end up with the same total amount of flow. This is solved by defining a flow $f_{tran,ij}$ that simply sends the flow across each edge in the boundary matching $\edgeh{H_{2i}}{H_{2j}}$.

The problems~$\pi_{rout,i,1}$ and~$\pi_{rout,i,2}$ are induced by~$f_{tran,ij}, j \in \{1,2\}$: $f_{tran,ij}$ transmits flow evenly \emph{among} $H_p^{n-2}$ copies in $H_2$, and~$f_{distsub,i}$ (once we construct it) will distribute flow evenly \emph{within} each subgraph. The problem~$\pi_{rout,i,1}$ (respectively~$\pi_{rout,i,2}$) is that of routing, through~$H_{2,i}$, for~$3 \leq i \leq p$, the flow from~$H_1$ that is bound for~$H_{2,1}$ (respectively~$H_{2,2}$).

We now prove \cref{lem:fdistrecdef}:

\begin{proof}[(Proof of \cref{lem:fdistrecdef})]
   The claim will follow from comparing source and sink sets and~$\sigma, \delta$ functions, using the following observations:
   
   The equality $\sigma_{distsub,i} = ((p-2)/p)\sigma_{dist}$, $3 \leq i \leq p$, comes from observing that $p-2$ of the $p$ $H_p^{n-2}$ subgraphs of $H_2$ receive flow from $H_1$, and thus vertices in $H_i$, $3 \leq i \leq p$, will, in order to solve $\pi_{dist}$, distribute a $1/p$ factor of the flow they have received to each of $H_1$ and $H_2$, and distribute the remaining $(p-2)/p$ factor internally.

    Since $\bdryh{H_2}{H_1} = \bigcup_{3 \leq i \leq p} \faceth{12}(H_{2,i})$, we have by Remark~\ref{rmk:facetdecomp} that the flow received from~$H_1$ by $\bdryh{H_2}{H_1}$ is equally distributed \emph{among} (though certainly not \emph{within}) the $p - 2$ subgraphs $H_{2,3}, \dots, H_{2,p}$. Thus it suffices to let $H_{2,3}, \dots, H_{2,p}$ each distribute a $1/p$ factor of this flow to each of~$H_{2,1}$ and~$H_{2,2}$, justifying $\pi_{tran,i,1}$ and $\pi_{tran,i,2}$.

    For the ``routing'' problems $\pi_{rout,i,1}$ and $\pi_{rout,i,2}$, within each~$H_{2,i}$, the flow from~$H_1$ is initially concentrated within~$\faceth{12}(H_{2,i})$ (since, again, this facet forms the restriction of the $H_1, H_2$ boundary~$\bdryh{H_2}{H_1}$ to~$H_{2,i}$). The flow needs to be sent to~$\bdryh{H_{2,i}}{H_{2,1}} = \faceth{1,i}(H_{2,i})$ and~$\bdryh{H_{2,i}}{H_{2,2}} = \faceth{2,i}(H_{2,i})$\textemdash so that~$f_{tran,i,1}$ and~$f_{tran,i,2}$ can send the flow to~$H_{2,1}$ and~$H_{2,2}$.

    Furthermore, $|\faceth{12}(H_{2,i})| = |\faceth{1,i}(H_{2,i})| = |\faceth{2,i}(H_{2,i})|$ by \cref{rmk:facetdecomp}, so $\delta_{rout,i,1} = \sigma_{rout,i,1}$ and $\delta_{rout,i,2} = \sigma_{rout,i,2}$. 

    This justifies the source, sink, surplus, and demand functions in~$\pi_{rout,i,1}$ and~$\pi_{rout,i,2}$.

    Finally,~$H_{21}$ (respectively~$H_{22}$) receives~$(1/p)\sigma_{dist}$, at each vertex in~$\bdryh{H_{2,1}}{H_{2,i}}$ (respectively $\bdryh{H_{2,2}}{H_{2,i}}$), for $3 \leq i \leq p$; this flow must be distributed to the rest of~$V(H_{2,1})$ (respectively~$V(H_{2,2})$, justifying the definitions of~$\pi_{distsub,1}$ and~$\pi_{distsub,2}$.
\end{proof}

\begin{remark}\label{rmk:fdistsubrec}
Each subproblem~$\pi_{distsub,i}, 3 \leq i \leq p,$ is a problem of the same form as~$\pi_{dist}$, over a Hanoi graph isomorphic to $H_p^{n-2}$ instead of $H_p^{n-1}$. That is, $\pi_{distsub,i}$ is the problem of distributing flow, initially concentrated uniformly within a facet of a Hanoi subgraph, throughout the Hanoi subgraph.
\end{remark}

\begin{remark}\label{rmk:fdistsubrecj}
Each of~$\pi_{distsub,1}$ and~$\pi_{distsub,2}$ is the sum of~$p-2$ recursive subproblems of the same form as~$\pi_{dist}$.
\end{remark}

We now have a complete recursive decomposition of~$\pi_{dist}$\textemdash originally defined within~$H_2 \cong H_p^{n-1}$\textemdash into subproblems on its~$H_p^{n-2}$ subgraphs.

(The base case occurs when $n = 2$, and $H_p^{n-1}$ is a clique.)

This decomposition, pending the definition of $f_{tran,ij}$, yields a recursive definition of~$f_{dist}$.

We now bound the resulting congestion:

\begin{lemma}\label{lem:ftranrec}
The subproblem~$\pi_{tran,i,j}$ can be solved with congestion~$\frac{\sigma_{dist}}{p|V(H_p^n)|}$.
\end{lemma}
\begin{proof}
    We define~$f_{tran,ij}$ as follows: simply send $\sigma_{tran,ij} = \frac{\sigma_{dist}}{p}$ flow across each edge in the matching $\edgeh{H_{2,i}}{H_{2,j}}$.

    Now,~$f_{tran,ij}$ produces at most the desired congestion when normalized by~$|V(H_p^n)|$.
\end{proof}

\begin{lemma}\label{lem:fdistroutrec}
The subproblem~$\pi_{rout,i,j}$ can be solved with congestion~$\frac{\sigma_{rout,i,j}}{|V(H_p^n)|}$.
\end{lemma}
\begin{proof}
    Each~$f_{tran,ij}$ flow induces the two~$\pi_{rout,i,j}$ subproblems, each of which requires routing~$\sigma_{tran,i,j} = \delta_{tran,i,j} = \sigma_{rout,i,j}$ flow from one facet to another within~$H_{2,i}$.

Furthermore,~$\pi_{rout,i,1}$ is the problem, within~$H_{2,i} \cong H_p^{n-2}$, of sending some amount of flow from a source set that is a facet, to a sink set that is a facet of the same size. The surplus and demand functions are uniform over their domains. Remark~\ref{rmk:facetdecomp} gives an immediate decomposition of~$\pi_{rout,i,1}$ into subproblems within and among the~$p$ copies of~$H_p^{n-3}$ that comprise~$H_{2,i}$. The subproblems \emph{within} the $H_p^{n-3}$ subgraphs are facet-to-facet routing problems of the same form as~$\pi_{rout,i,1}$; the subproblems \emph{among} the subgraphs are transmission problems of the same from as~$\pi_{tran,ij}$.

    We need to make sure that the routing subproblems $\pi_{rout,i,j}$ within $H_{2,i}, 3 \leq i \leq p, j \in \{1,2\},$ do not result in a compounding of congestion in the recursion. Within~$H_{2,i}$, we have the problem of routing flow from some source facet, say $\mathcal{F}_q \subseteq H_{2,i}$, to some sink facet, say $\mathcal{F}_r \subseteq H_{2,i}$, such that $|\mathcal{F}_q| = |\mathcal{F}_r|$, and such that $\sigma_{rout,i,j} = \delta_{rout,i,j} = \sigma_{tran,i,j}$. Remark~\ref{rmk:facetdecomp} implies that $|\mathcal{F}_q|$ and $|\mathcal{F}_r|$ each decompose into a union of $p - 2$ facets, of equal cardinality, within $p - 2$ of the $H_p^{n-3}$ subgraphs of $H_{2,i}$. Thus it suffices to:
    \begin{enumerate}
    \item consider the recursive routing subproblem, of identical form to $\pi_{rout,i,j}$, within each $H_p^{n-3}$ subgraph having a nonempty intersection with both $\mathcal{F}_q$ and $\mathcal{F}_r$, and
    \item find an arbitrary matching between $H_p^{n-3}$ subgraphs intersecting only with $\mathcal{F}_q$ and those intersecting only with $\mathcal{F}_r$. (Such a matching always exists.)
    \end{enumerate}
    
    The recursive subproblems in (1) each have surplus and demand values equal to $\sigma_{rout,i,j} = \delta_{rout,i,j}$. For (2), after finding the matching, we have a transmission problem between each source-sink pair of Hanoi subgraphs in the matching, and a routing problem within each source subgraph and each sink subgraph, again with surplus and demand equal to $\sigma_{rout,i,j} = \delta_{rout,i,j}$. 

    As desired, this avoids any recursive gain, and the resulting (normalized) congestion is~$\frac{\sigma_{rout,i,j}}{|V(H_p^n)|}$.

\end{proof}

\begin{lemma}
\label{lem:hanpartialcong}
There exists an MSF~$f_{dist}$ that solves the MSF problem~$\pi_{dist}$ while producing at most~$\frac{\sigma_{dist}}{|V(H_p^n)|}$ congestion across each edge in~$H_2$.
\end{lemma}
\begin{proof}

The MSF~$f_{dist}$, given the decomposition, is entirely specified by the definitions we have given for the flows solving each subproblem.

We now apply induction, on $n$, to the recursive definition of~$f_{dist}$, to bound the congestion. The inductive hypothesis is the assumption that each~$\pi_{distsub,i}$ can be solved while producing at most~$\sigma_{distsub,i}$ (non-normalized) flow across each edge.

The distribution subproblems~$\pi_{distsub,i}$ then produce $\sigma_{distsub,i} = ((p-2)/p)\sigma_{dist}$ combined flow across edges within each~$H_{2,i}$,  $3 \leq i \leq p$\textemdash this is immediate from the inductive hypothesis.

Each of the subproblems $\pi_{distsub,1}$ and $\pi_{distsub,2}$ naturally decomposes (\cref{rmk:fdistsubrecj}) into $p-2$ subproblems, one for each of the boundary sets $\delta_v(H_{2,1}, H_{2,i})$, $i = 3, \dots, p$ and similarly one for each of the boundary sets $\delta_v(H_{2,2}, H_{2,i})$, $i = 3, \dots, p$. This produces total flow $\frac{(p-2)\sigma_{dist}}{p}$ across edges in~$H_{2,1}$ and~$H_{2,2}$ respectively.

Now, to bound the congestion resulting from~$\pi_{rout,i,j}$, recall that $\sigma_{rout,i,j} = \delta_{rout,i,j} = \sigma_{tran,i,j} = (1/p)\sigma_{dist}$. Applying \cref{lem:fdistroutrec} and summing over $j \in \{1, 2\}$, gives total (non-normalized) flow
$$(2/p)\sigma_{dist}$$
from the routing MSFs within each $H_{2,i}$. 

Adding the~$\pi_{rout,i,j}$ bound of~$(2/p)\sigma_{dist}$ to the $f_{distsub,i}$ congestion bound of $((p-2)/p)\sigma_{dist}$ shows that~$f_{dist}$ produces total congestion at most $\sigma_{dist}$ within each~$H_{2,i}$, $3 \leq i \leq p$, as claimed.

The claim now follows from normalizing the above congestion bounds by $|V(H_p^n)|$.
\end{proof}

We can now prove Lemma~\ref{lem:hanoiindflow}:
\lemhanoiindflow*
\begin{proof}
The claim follows from applying Lemma~\ref{lem:hanpartialcong} to the $p - 1$ $H_p^{n-1}$ subgraphs that need to send flow to vertices in $H_2$, and doing the same for each other $H_i$.
\end{proof}

Theorem~\ref{thm:hanoiexplb} follows via induction on~$n$.

\bibliographystyle{abbrvnat}
\bibliography{references}

@inbook{FPTsurvey,
author = {Bodlaender, Hans L.},
chapter = {Fixed-Parameter tractability of treewidth and pathwidth},
year = {2012},
title = {The Multivariate Algorithmic Revolution and Beyond: Essays Dedicated to Michael R. Fellows on the Occasion of His 60th Birthday},
publisher = {Springer-Verlag}
}

@misc{ericksontw,
author = {Jeff Erickson},
title = {Computational Topology: Treewidth},
howpublished = {Lecture Notes},
url = {http://jeffe.cs.illinois.edu/teaching/comptop/2009/notes/treewidth.pdf},
year = {2009}
}

@inproceedings{anari2,
author = {Anari, Nima and Liu, Kuikui and Gharan, Shayan Oveis and Vinzant, Cynthia},
title = {Log-Concave Polynomials {II}: High-Dimensional Walks and an {FPRAS} for Counting Bases of a Matroid},
year = {2019},
booktitle = {Proceedings of the 51st Annual ACM SIGACT Symposium on Theory of Computing (STOC 2019)},
doi = {10.1145/3313276.3316385}}

@incollection{kaibelexp,
  title={On the expansion of graphs of 0/1-polytopes},
  author={Kaibel, Volker},
  booktitle={The Sharpest Cut: The Impact of Manfred Padberg and His Work},
  pages={199--216},
  year={2004},
  publisher={SIAM},
  doi = {10.1137/1.9780898718805.ch13}
}

@inproceedings{mct,
  title={On the mixing time of the triangulation walk and other Catalan structures},
  author={Lisa McShine and P. Tetali},
  booktitle={Randomization Methods in Algorithm Design},
  year={1997},
  doi = {10.1090/dimacs/043}
}

@inproceedings{molloylb,
author = {Molloy, Michael and Reed, Bruce and Steiger, William},
year = {1997},
title = {On the Mixing Rate of the Triangulation Walk},
booktitle = {Randomization Methods in Algorithm Design}
}

@article{sinclair_1992, title={{Improved bounds for mixing rates of Markov chains and multicommodity flow}}, DOI={10.1017/S0963548300000390},  journal={Combinatorics, Probability and Computing}, author={Sinclair, Alistair}, year={1992}}

@article{jerrumprojres,
 URL = {http://www.jstor.org/stable/4140446},
 author = {Mark Jerrum and Jung-Bae Son and Prasad Tetali and Eric Vigoda},
 journal = {The Annals of Applied Probability},
 title = {{Elementary bounds on Poincaré and log-Sobolev constants for decomposable Markov chains}},
 year = {2004}
}

@book{levin2017markov,
  title={Markov chains and mixing times},
  author={Levin, David A and Peres, Yuval and Wilmer, Elizabeth},
  year={2017},
  publisher={American Mathematical Soc.},
  doi = {10.1090/mbk/058}
}

@article{hanoitw,
title = {On the treewidth of {Hanoi} graphs},
journal = {Theoretical Computer Science},
year = {2022},
doi = {10.1016/j.tcs.2021.12.014},
author = {David Eppstein and Daniel Frishberg and William Maxwell}
}

@InProceedings{eppfrisharx,
  author =	{Eppstein, David and Frishberg, Daniel},
  title =	{{Rapid Mixing for the Hardcore Glauber Dynamics and Other Markov Chains in Bounded-Treewidth Graphs}},
  booktitle =	{34th International Symposium on Algorithms and Computation (ISAAC 2023)},
  year =	{2023},
  doi =		{10.4230/LIPIcs.ISAAC.2023.30},
}

@InProceedings{eppfrishtri,
  author =	{Eppstein, David and Frishberg, Daniel},
  title =	{{Improved Mixing for the Convex Polygon Triangulation Flip Walk}},
  booktitle =	{50th International Colloquium on Automata, Languages, and Programming (ICALP 2023)},
  year =	{2023},
  doi =		{10.4230/LIPIcs.ICALP.2023.56},
}

@book{cartesian,
author = {Imrich, Wilfried and Klavžar, Sandi and Rall, Douglas F.},
year = {2008},
title = {Topics in Graph Theory: Graphs and Their Cartesian Product},
publisher = {A K Peters},
doi = {0.1201/b10613}
}

@book{mythsmaths,
  title={The Tower of Hanoi\textemdash Myths and Maths},
  author={Hinz, Andreas M. and Milutinovi\'c, Uro\u{s} and Klav{\v{z}}ar, Sandi and Petr, Ciril},
  year={2018},
  publisher={Springer},
  doi={10.1007/978-3-319-73779-9}
}

@article{harveywood,
  title={{Treewidth of the Kneser Graph and the Erd\H{o}s--Ko--Rado theorem}},
  author={Daniel J. Harvey and David R. Wood},
  journal={Electronic Journal of Combinatorics},
  year={2014},
  doi={10.37236/3971}
}

@article{genkneser,
  title={{Treewidth of the Generalized Kneser Graphs}},
  author={Ke Liu and Mengyu Cao and Mei Lu},
  journal={Electronic Journal of Combinatorics},
  year={2022},
  doi={10.37236/10035}
}

@article{arv,
  title={Expander flows, geometric embeddings and graph partitioning},
  author={Arora, Sanjeev and Rao, Satish and Vazirani, Umesh},
  journal={Journal of the ACM (JACM)},
  year={2009},
  publisher={ACM New York, NY, USA}
}

@article{babaiszegedy,
  title={Local expansion of symmetrical graphs},
  author={Babai, L{\'a}szl{\'o} and Szegedy, Mario},
  journal={Combinatorics, Probability and Computing},
  year={1992},
  publisher={Cambridge University Press}
}
\label{sec:biblio}

\end{document}